\documentclass[12pt]{amsart}
\usepackage{amsthm,vmargin}
\usepackage{amsmath,xypic}
\usepackage{devanagari}

\usepackage{graphicx}
\begin{document}

%@modernhindi

\newcommand{\Q}{{\mathbb Q}}
\newcommand{\C}{{\mathbb C}}
\newcommand{\R}{{\mathbb R}}
\newcommand{\Z}{{\mathbb Z}}
\newcommand{\F}{{\mathbb F}}
\newcommand{\bF}{{\bar\F}}
\renewcommand{\wp}{{\mathfrak p}}
\renewcommand{\P}{{\mathbb P}}
\renewcommand{\O}{{\mathcal O}}
\newcommand{\Pic}{{\rm Pic\,}}
\newcommand{\Ext}{{\rm Ext}\,}
\newcommand{\rank}{{\rm rk}\,}
\newcommand{\sbull}{{\scriptstyle{\bullet}}}
\newcommand{\bX}{X_{\overline{k}}}
\newcommand{\ch}{\operatorname{CH}}
\newcommand{\tors}{\text{tors}}
\newcommand{\cris}{\text{cris}}
\newcommand{\alg}{\text{alg}}
\newcommand{\tX}{{\tilde{X}}}
\newcommand{\tL}{{\tilde{L}}}
\newcommand{\Hom}{{\rm Hom}}
\newcommand{\spec}{{\rm Spec}}
\newcommand{\gal}{{\rm Gal}}
\renewcommand{\int}{\operatorname{int}}
\let\isom=\simeq
\let\rk=\rank
\let\tensor=\otimes
\newcommand{\X}{\mathfrak{X}}
\newcommand{\M}{\mathcal{M}}
\newcommand{\A}{\mathcal{A}}
\newcommand{\N}{\mathcal{N}}
\newcommand{\Jac}{\rm{Jac}}
\newcommand{\End}{\rm{End}}
\newcommand{\U}{\mathcal{U}}
\newcommand{\Ub}{\bar{\mathcal{U}}}
\newcommand{\mydot}{{\small{\bullet}}}
\renewcommand{\go}{{\rm GO}}
\newcommand{\pgo}{{\rm PGO}}
\newcommand{\disc}{\mathop{{\rm disc}}}
\newcommand{\Nrd}{{\rm Nrd}}
\let\GO=\go
\let\PGO=\pgo
\newcommand{\Gspin}{{\rm GSpin}}
\renewcommand{\O}{{\rm O}}
\newcommand{\spin}{{\rm Spin}}
\let\Spin=\spin
\let\SPin=\spin
\let\GSpin=\Gspin
\let\gspin=\Gspin
\newcommand{\Gm}{\mathbb{G}_m}
\let\into=\hookrightarrow
\let\congruent=\equiv
\newcommand{\frob}{{\rm Frob}}
\newcommand{\hasse}[3]{\left({{#1},{#2}\over\Q_{#3}}\right)}
\newcommand{\hassef}[3]{\left({{#1},{#2}\over F_{#3}}\right)}
\newcommand{\Spec}{{\rm Spec}}

\newcommand{\Trd}{{\rm Trd}}

\newtheorem{theorem}[equation]{Theorem}      % (If you want theorem numbered
\newtheorem{lemma}[equation]{Lemma}          %
\newtheorem{corollary}[equation]{Corollary}  %       goes for lemmas, etc.)
\newtheorem{proposition}[equation]{Proposition}
\newtheorem{scholium}[equation]{Scholium}

\theoremstyle{definition}
\newtheorem{conj}[equation]{Conjecture}
\newtheorem*{example}{Example}
\newtheorem{question}[equation]{Question}

\theoremstyle{definition}
\newtheorem{remark}[equation]{Remark}

\numberwithin{equation}{subsection}

\title{Musings on $\Q(1/4)$: Arithmetic spin structures on elliptic curves}
\author{Kirti Joshi}
\address{Math. department, University of Arizona, 617 N Santa Rita, Tucson
85721-0089, USA.} \email{kirti@math.arizona.edu}
\date{}

\begin{abstract}
We introduce and study arithmetic spin structures on
elliptic curves. We show that there is a unique isogeny class of
elliptic curves over $\F_{p^2}$ which carries a unique arithmetic
spin structure and provides a geometric object of weight $1/2$ in
the sense of Deligne and Grothendieck. This object is thus a
candidate for $\Q(1/4)$.
\end{abstract}

\maketitle

%\begin{flushright}
%\includegraphics[scale=.35]{madhushala3.png}
% if needed translation of the quartet is here:
% daily I distill this sweet wine: spiced with my imagination
% daily I fill the parched goblet of my being
% and drink this brimming cup
% for I am the saqi, the reveler and the tavern.
%               Harivanshrai `Bacchan'.

\footnote{From the ``Madhush\=al\=a'' by Harivansh R\=ai Bachchan.}

\begin{flushright}%{\it
\includegraphics[scale=.35]{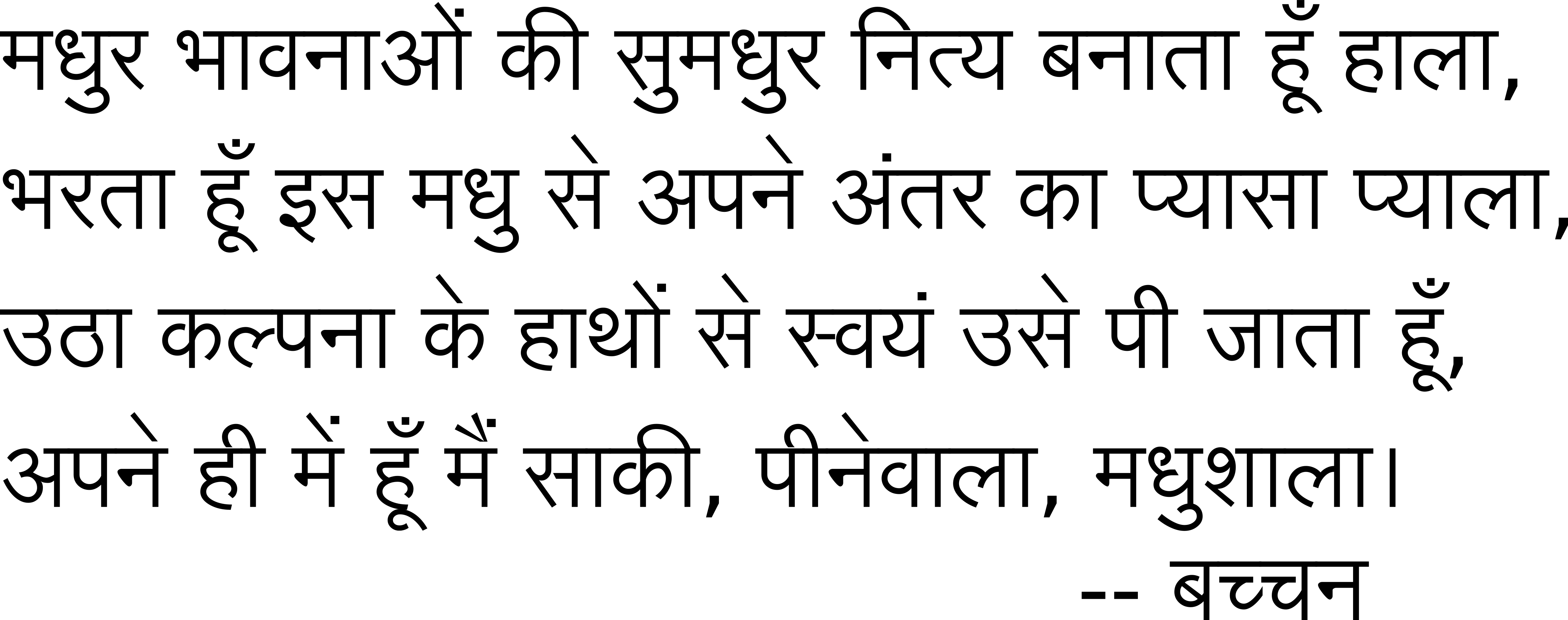}
\end{flushright}

\tableofcontents

\section{Introduction and statement of the principal results}\label{sec:Introdu}
\subsection{The problem}
In this paper we attempt to provide an intrinsic approach to the
problem of constructing $\Q(1/4)$. We show that there is, up to isomorphism,
a unique geometric object which lives over $\F_{p^2}$,
and which is equipped with $\ell$-adic and $p$-adic realizations of
weight $1/2$ in the sense of \cite{deligne80}. This object is thus a
candidate for $\Q(1/4)$. This object is constructed using what we
call \textit{arithmetic spin structures on elliptic curves}. Our
approach to $\Q(1/4)$ may be considered to be a twisted analogue of
the constructions of \cite{kuga67,deligne72}.  We note that
the problem of constructing ``fractional motives'' and ``exotic'' Tate motives
has also been studied in \cite{anderson86,manin92,deninger94, ramachandran05}.

We say that an elliptic curve is of \textit{spinorial type} if its
algebra of endomorphisms (we consider endomorphisms defined over the
field of definition of the curve) carries a non-trivial involution
of the \textit{first kind} (see
\ref{ssec:Ellipti-curves-spinori-type}). The classification of
endomorphism algebras of elliptic curves shows that an elliptic
curve is of spinorial type if and only if it is supersingular and
its endomorphism algebra is a quaternion algebra (see
Proposition~\ref{pro:EhasSpin}). Thus a supersingular elliptic curve
with all its endomorphisms defined over the ground field is of
spinorial type; and every supersingular elliptic curve becomes of
spinorial type over some quadratic extension. By the well-known
classification of isogeny classes of elliptic curves over a finite
field one sees that for a given finite field there are at most two
isogeny classes of elliptic curves which are of spinorial type.

Involutions of the first kind on an algebra are  classified as
\textit{orthogonal or symplectic} depending on what they look like
over any splitting field of the algebra. A \textit{spin structure}
$(B,\sigma)$ on an elliptic curve of spinorial type is a choice of
an involution $\sigma$ of the first kind and orthogonal type (i.e of
an orthogonal involution--see \ref{ssec:Orthogo-symplec-involut}) on
the endomorphism algebra (denoted here by $B$) of the curve. On a
quaternion algebra, orthogonal involutions of the first kind are
classified (up to isomorphism) by their discriminant. Every spin
structure comes equipped with a (even) Clifford algebra (see
\ref{ssec:Cliffor-Algebra-associa-pair}). In the present situation,
because the quaternion algebra $B$ of endomorphisms of a
supersingular elliptic curve is ramified  at infinity (and at $p$,
the characteristic of ground field), the (even) Clifford algebras
$C^+(B,\sigma)$ which can arise are imaginary quadratic extensions
of $\Q$ (see Theorem~\ref{the:existenceofarithmetic}).

\subsection{The results over $\F_{p^2}$} We now describe the results we obtain over $\F_{p^2}$
(they are also proved here for $\F_{p^{2n}}$ with $n$ odd). We show
(in Proposition~\ref{frobeniussimilitude}) that the Frobenius
endomorphism of $E$ is a \textit{proper similitude} of this spin
structure, and the \textit{group of proper similitudes}, denoted
here by $\go^+(B,\sigma)$, is the non-split torus (obtained by the
restriction of scalars of $\Gm$ from the Clifford algebra
$C^+(B,\sigma)$ of the spin structure--see
Proposition~\ref{pro:goplus}). The \textit{spin group (rather the
general spin group)} associated to $(B,\sigma)$ is again a non-split
torus and is a cover of the group of similitudes (see
\ref{ssec:GSpin}). This allows us to speak of constructing square
roots of the Frobenius endomorphism. A necessary and sufficient
condition that Frobenius endomorphism have a square root in the
Clifford algebra is that the Clifford algebra of the spin structure
is $\Q[x]/(x^2+p)$ (Theorem~\ref{existence-criterion}). This means
that the involution underlying the spin structure has discriminant
$-p$; moreover as the Frobenius endomorphism is a central element
operating by $\pm p$, one sees that there are spin structures for
which Frobenius does not have a square root.

We show (see Theorem~\ref{the:existenceofarithmetic} and
Corollary~\ref{cor:existenceoverfp2}) that for an elliptic curve
$E/\F_{p^2}$ of spinorial type, whose Frobenius endomorphism acts by
multiplication by $-p$, carries a unique spin structure of
discriminant $-p$ (uniqueness is up to isomorphism). For this spin
structure the Frobenius endomorphism admits square roots
($\pm\sqrt{-p}$) in the Clifford algebra. This gives rise to a
spinorial representation of the Weil group
$$\rho_{E,\sigma}^{spin}:W(\bF_{p^2}/\F_{p^2})\to \gspin(B,\sigma)$$ (here
$B=\End(E)$) which lifts the canonical \textit{similitude representation}
$$\rho_{E,\sigma}:W(\bF_{p^2}/\F_{p^2})\to\GO^+(B,\sigma).$$  This spinorial representation is of weight
$1/2$ in the sense that the absolute value of the eigenvalues of
Frobenius is $$\sqrt{p}=(p^2)^{1/4}=(p^2)^{(\frac{1}{2})/2}.$$ Spin
structures for which the similitude representation admits a
spinorial lifting are said to be \textit{arithmetic spin
structures}. Over $\F_{p^2}$, arithmetic spin structures of
discriminant $-p$ can exists on elliptic curve of spinorial type if
and only if the Frobenius endomorphism operates by $-p$; and a
fortiori, there are no orthogonal involutions on the algebra
$B=\End(E)$ with positive discriminant (where $E$ is of spinorial
type). So not all spin structures can exist or when they exist are
arithmetic. The data  $(E,(B,\sigma))$ consisting of elliptic curve
$E/\F_{p^2}$ together with  an arithmetic spin structure
$(B,\sigma)$ on $E$  is a geometric object whose weight is half:
this our candidate for $\Q(1/4)$. There is exactly one isogeny class
of elliptic curves over $\F_{p^2}$ which provides such a structure
and the spin structure is unique up to isomorphism. In
\ref{ell-realization} and \ref{p-realization} we construct the
$\ell$-adic and the crystalline realizations of $\Q(1/4)$. We note
that the Clifford algebra of spin structures (arithmetic or not) are
imaginary quadratic extensions of $\Q$ (and so are not isomorphic to
$\Q\times\Q$).

In \cite{ramachandran05} one finds  a formula for the values of the
Hasse-Weil zeta function of a smooth, projective variety over $\F_{p^2}$ at $s=1/2$.
It has been expected that this formula must have an arithmetic
explanation in terms of a motive $\Z(1/2)$. As a consequence of our theory
we prove a simple relation (see
Theorem~\ref{l-identities}) between the $L$-function of the elliptic
curve with an arithmetic spin structure and the associated spinorial
representation. This is
similar (in spirit) to the formula proved in \cite{ramachandran05}.

The theory developed here should be viewed as a $\Q$-theory rather
than a $\Z$-theory because we have worked throughout with the
quaternion algebra (so up to isogeny).  But for a $\Z$-theory one
should work with orders in the quaternion algebra, as endomorphism
rings arise more naturally as  orders). We hope to return to this in
later paper. The theory developed here can be also applied to rank
two supersingular Drinfeld modules, but we will defer this to a
subsequent paper as well.

\subsection{Acknowledgements} This paper grew out of a talk given by Ramachandran on
\cite{ramachandran05}. We thank him for conversations and
correspondence about \cite{ramachandran05}. Thanks are also due to
Preeti Raman for conversations on Clifford groups and for help with
identifying $\gspin(B,\sigma)$ in the case we need here. Thanks are
also due to Dinesh Thakur for encouragement and for a careful
reading of the manuscript. We would like to thank Christopher
Deninger for comments and suggestions which have improved this manuscript. Some part of this work was
carried out while the author was visiting the Tata Institute and the
Universit\'{e} Montpellier II and we thank both for their
hospitality.

\section{Recollections from the theory of Clifford
algebras}\label{sec:Recolle-from-theory-Cliffor-algebra}
\subsection{Preparatory remarks}
The fundamental reference for what we need here is \cite{knus-book}.
Readers unfamiliar with the theory of Clifford algebras may first
want to read \cite{chevalley-spinors} but should bear in mind that
unlike \cite{chevalley-spinors} we will work with a twisted
situation. All facts we need about Clifford algebras constructed
from quaternion algebras with orthogonal involutions are found in
\cite{knus-book}. The reader is advised to keep that work handy
while reading the present paper. Since the machinery of twisted
quadratic spaces may not be familiar to the readers we have, for the
reader's convenience, inserted a ``recurring example'' which
explains the machinery in the context we need to use for the main
results of this paper.

Throughout the paper $F$ will
denote a field of characteristic not equal to two (zero is allowed).
The results can probably be carried out in case the characteristic
is two but the details are more complicated so we will leave them
aside for the moment.

\subsection{Inner automorphisms} In what follows, we will work with finite dimensional algebras over a field $F$.
Let $A$ be a finite dimensional algebra  over a field $F$. For an
invertible $u\in A$, we write $\int(u):A\to A$ for the inner
automorphism of $A$ defined by $u$, given explicitly by
$\int(u)(a)=u\circ a\circ u^{-1}$ for any $a\in A$.

\subsection{Involutions}
Let $F$ be a field of characteristic not equal to two. Let $A/F$ be
a finite dimensional algebra over $F$. An \emph{involution of the first
kind} on $A$ is a map $\sigma:A\to A$ such that $\forall x,y\in A$,
\begin{enumerate}
\item  $\sigma(x+y)=\sigma(x)+\sigma(y)$,
\item $\sigma(xy)=\sigma(y)\sigma(x)$ (so $\sigma$ is an anti-morphism)
\item $\sigma^2(x)=x$,
\item $\sigma$ is identity on the center of $A$.
\end{enumerate}

\subsection{Classification of involutions}\label{ssec:Constru-involut}
\subsubsection{The split case} Let $V$ be a finite dimensional vector
space over a field extension $K/F$ and let $b:V\times V\to K$ be a
non-degenerate bilinear form with values in $K$. Then for any
$f\in\End_K(V)$ we define $\sigma_b(f)\in\End_K(V)$ by the following
property: for all $v,w\in V$ we have
\begin{equation*}
b(v,f(w))=b(\sigma_b(f)(v),w).
\end{equation*}
Then $f\mapsto \sigma_b(f)$ is a $K$-linear  anti-automorphism of
$\End_K(V)$, called the adjoint automorphism of $\End_K(V)$ with
respect to the non-degenerate bilinear form $b$. The mapping
$b\mapsto \sigma_b$ is a bijection between the equivalence classes
of non-degenerate bilinear forms on $V$ up to scalar multiples and
the set of $K$-linear anti-automorphisms of $\End_K(V)$. Under this
equivalence $\sigma_b$ provides an involution (i.e. a $K$-linear
anti-automorphism of order two) if and only if $b$ is symmetric or
skew-symmetric. For a proof see \cite[Page 1]{knus-book}.

\subsubsection{The general case} We now describe involutions $\sigma$
on arbitrary central simple algebras. Let $(A,\sigma)$ be a pair
where $A/F$ is a central simple algebra and $\sigma:A\to A$ is an
involution. Let $K/F$ be a splitting field of $A/F$. Then
$\sigma_K:A\tensor_FK\to A\tensor_F K$ is an involution of $A_K$ and
by the previous paragraph, we see that $\sigma_K$ arises from a
non-singular bilinear form which is either symmetric or
skew-symmetric.

\subsection{Orthogonal and symplectic
involutions}\label{ssec:Orthogo-symplec-involut}
Let $(A,\sigma)$ be a pair as above. We say that $\sigma$ is
\emph{symplectic} (resp. \emph{orthogonal}) involution if for any
splitting field $K/F$ (and any isomorphism $A_K\to\End_K(V)$, the
involution $\sigma_K$ of $A_K$ arises from a non-degenerate
skew-symmetric (resp. symmetric) bilinear form on $V$.

\subsection{Twisted quadratic spaces}\label{ssec:Twisted-quadrat-spaces}
Let $F$ be a field (of characteristic not equal to two). A
\emph{twisted quadratic space} over $F$ is a pair $(A,\sigma)$ where
$A/F$ is a central simple algebra over $F$ and $\sigma: A\to A$ is
an involution of first kind and of orthogonal type. Morphisms of
twisted quadratic spaces are defined in the obvious way.

\subsection{Twisted quadratic spaces and quadratic spaces}
If $(A,\sigma)$ is a twisted quadratic space over $F$ and if $A$ is
split with $A=\End_F(V)$, then $\sigma=\sigma_b$ for a symmetric
bilinear form $b:V\times V\to F$ and the pair $(A,\sigma)$ is
isomorphic to the pair $(\End_F(V), \sigma_b)$ and this is
equivalent to the data $(V,q_b)$ where $q_b: V\to F$  is the
quadratic form associated to the symmetric bilinear form $b$.  Thus
in the split case the data of a twisted quadratic space is simply
the data of a quadratic space. We note that the term ``twisted
quadratic space'' was not introduced in \cite{knus-book} but clearly
seems appropriate.

\subsection{Recurring Example}\label{ssec:Recurri-Example}
This example will recur throughout and is the case we want to consider for the main results of
this paper. So we provide this along for the convenience of the reader.

Let $B/F$ be a quaternion algebra. Then $B$ has a basis $1,i,j,k$
with $i^2,j^2\in F^*$, $ij=k=-ji$. If $i^2=a,j^2=b$ we will write
this algebra as $B=\hassef{a}{b}{}$. The map $\gamma:B\to B$ given
by sending $$\gamma(x_0+ix_1+jx_2+kx_3)= x_0-ix_1-jx_2-kx_3$$ is the
unique symplectic involution on $B$ (see \cite[Proposition~2.21, page 26]{knus-book}).

Every orthogonal involution $\sigma$ on $B$ is of the form
$$\sigma=\int(u)\circ\gamma$$ where $\gamma(u)=-u$ and $u\in B^*$ is
uniquely determined up to a scalar in $F^*$ by $\sigma$ (see \cite[Proposition~2.21, page 26]{knus-book}).

By \cite[Corollary~2.8 (page 18) and Proposition~2.21 (page 26)]{knus-book}
every quaternion algebra $B/F$ carries both symplectic and
orthogonal involutions. The reduced norm of $u\in B$ is given
$\Nrd(u)=u\gamma(u)$, the reduced trace of $u\in B$ is given by
$\Trd(u)=u+\gamma(u)$. Observe that $u$ has reduced trace zero if
and only if $\gamma(u)=-u$ (i.e., $u$ is a ``pure quaternion'').

\subsection{Discriminants}\label{ssec:Discrim}
\subsubsection{The split case} We first describe the discriminant of
a quadratic space. Let $(V,q)$ be a non-degenerate quadratic space
over $F$. Let $b$ be the associated bilinear form. For a basis
$e_1,\ldots, e_n$ of $V$, the matrix $\det(b(e_i,e_j))\neq0$ and its
class in $F^*/F^{*2}$ is independent of the choice of the basis
$e_1,\ldots,e_n$ of $V$. We denote this class in $F^*/F^{*2}$ by
$\disc(b)$. The discriminant of $(V,q)$ is the class
$$\disc(V,q)=\disc(q)=(-1)^{\frac{n(n-1)}{2}}\det(b)\in F^*/F^{*2}.$$

\subsubsection{The general case}\label{sssec:disc-general-case}
For an even degree central simple algebra $A$ and
any orthogonal involution $\sigma$, we define the determinant
$$\det(\sigma)=\Nrd_A(a)\in F^*/F^{*2},$$ for any element $a\in A^*$
such that $\sigma(a)=-a$. This class is again independent of the
choice of such an $a$. The discriminant of $\sigma$ is given by
$$\disc(\sigma)=(-1)^\frac{\deg(A)}{2}\det(\sigma)\in F^*/F^{*2}.$$
If $\sigma=\int(u)\circ\gamma$ then
$$\disc(\sigma)=-\Nrd_A(u)\in F^*/F^{*2}.$$ For a proof
see \cite[Page 81, Proposition 7.3(2)]{knus-book}. We note that the
sign given in that reference is not correct (ours is correct) as can
be easily seen from the proof given on \cite[Pages 81-82,
Proposition 7.3]{knus-book}.

Moreover if $A$ is a quaternion division algebra then $A$ does not
carry any orthogonal involutions with trivial discriminant (see
\cite[page 82, Example 7.4]{knus-book}).

\subsection{Clifford Algebra associated to a twisted quadratic space
$(A,\sigma)$}\label{ssec:Cliffor-Algebra-associa-pair} Let
$(A,\sigma)$ be a twisted quadratic space. Then there exits a (even)
Clifford algebra denoted by $C^+(A,\sigma)$ which is functorial in
the pair and if $A=\End(V)$ for a vector space $V/F$, then
$C^+(A,\sigma)$ agrees with the \textit{even} clifford algebra constructed in
the usual manner. See \cite[page 91]{knus-book}.

\subsection{The center of the Clifford
algebra}\label{ssec:center-Cliffor-algebra} Let us assume from now
on that $(A,\sigma)$ is a twisted quadratic space over $F$ with
$\deg(A)=2m$, $m\geq 1$. The main case of interest to us will be the
case $m=1$, though we will not assume this preferring to work out
the general theory instead and specializing when we need to do so.

The center $Z(A,\sigma)$ of $C^+(A,\sigma)$ is an \'etale
(=separable) quadratic $F$-algebra. If $Z$ is a field, then
$C^+(A,\sigma)$ is a central simple algebra of degree $2^{m-1}$ over
$Z$; if $Z=F\times F$, then $C^+(A,\sigma)$ is a direct product of
two copies of central simple $F$-algebras of degree $2^{m-1}$.
Moreover the center $Z$ of $C^+(A,\sigma)$ is given by the following
recipe (see \cite[Theorem 8.10, page 94]{knus-book}).
\begin{theorem}\label{center-and-disc1}
Let $(A,\sigma)$ be a twisted quadratic space over $F$. Let
$C^+(A,\sigma)$ be the associated even Clifford algebra. Let
$Z=Z(A,\sigma)\subset C^+(A,\sigma)$ be its center. If the
characteristic of $F$ is not two, then $Z=F[X]/(X^2-\delta_\sigma)$ where
$\delta_\sigma\in F^*$ is a representative of the discriminant of $\sigma$,
$disc(\sigma)\in F^*/F^{*2}$.
\end{theorem}
\begin{corollary}\label{center-and-disc2}
Let $B/F$ be a quaternion algebra and $\sigma$ be an orthogonal
involution on $B$. Then the even Clifford algebra $C^+(B,\sigma)$ is commutative and we have
$$C^+(B,\sigma)=Z=F[X]/(X^2-\delta_\sigma),$$
where $\delta_\sigma=disc(\sigma)\mod{F^{*2}}$.
\end{corollary}
\subsection{Similitudes and the group of
similitudes}\label{ssec:Similit-group-similit} Let $(A,\sigma)$ be a
twisted quadratic space over $F$ as before. We study several groups
which arise in the present context.

A \emph{similitude} of $(A,\sigma)$ is an element $g\in A$ such that
$\sigma(g)g\in F^*$. Then $\mu(g)=\sigma(g)g$ is called the
multiplier of the similitude $g$ of $(A,\sigma)$. Similitudes of
$(A,\sigma)$ form a subgroup of $A^*$ which we denote by
$\go(A,\sigma)$. We have a homomorphism $\mu:\go(A,\sigma)\to F^*$
given by $g\mapsto \mu(g)=\sigma(g)g$. We define
$\pgo(A,\sigma)=\go(A,\sigma)/F^*$ and we have the exact sequence
$$1\to F^*\to \go(A,\sigma)\to \pgo(A,\sigma) \to 1.$$

Similitudes $g\in \go(A,\sigma)$ with $\mu(g)=1$ are called
\emph{isometries} and we have a subgroup
$O(A,\sigma)=\ker(\mu)\subset\go(A,\sigma)$ of isometries of
$(A,\sigma)$. We have an exact sequence of algebraic groups
$$1\to O(A,\sigma)\to \go(A,\sigma)\to \Gm\to 1.$$

\subsection{Recurring Example}\label{ssec:Recurri-Example2}
Let $(B,\sigma)$ be a twisted quadratic space over $F$ with $B$ a quaternion
division algebra (see \ref{ssec:Recurri-Example}).
Let  $\sigma=\int(u)\circ\gamma$ where $u\in B^*$ is a pure quaternion (so $\gamma(u)=-u$) and
$\gamma$ the canonical symplectic involution. Then $\go(A,\sigma)=F(u)^*\cup F(u)^*v$ where $v$
is an invertible quaternion which anti-commutes with $u$ (i.e. $uv=-vu$).

\subsection{Proper similitudes}\label{ssec:Proper-similit}
Let $\deg(A)=n=2m$. For every $g\in \go(A,\sigma)$ we have
$\Nrd_A(g)=\pm\mu(g)^m$, where $\Nrd(g)$ is the reduced norm of $g$.
We say that $g$ is a \emph{proper similitude} of $(A,\sigma)$ if
$\Nrd_A(g)=\mu(g)^m$. The set of proper similitudes of $(A,\sigma)$
is a subgroup of $\go(A,\sigma)$ denoted by $\go^+(A,\sigma)$. We
let $\pgo^+(A,\sigma)=\go^+(A,\sigma)/F^*$ and
$O^+(A,\sigma)=\go^+(A,\sigma)\cap O(A,\sigma)$ be the subgroup of
proper isometries of $(A,\sigma)$.

\subsection{Recurring Example}\label{ssec:Recurri-Example3}
Let $(B,\sigma)$ be a twisted quadratic space, with $B/F$ a quaternion division algebra
(see \ref{ssec:Recurri-Example},  \ref{ssec:Recurri-Example2}).
Suppose that $\sigma=\int(u)\circ\gamma$, let $N_{F(u)/F}:F(u)^*\to F^*$ be the norm map. Then we
have $$\go^+(B,\sigma)=F(u)^*,$$ and $O^+(B,\sigma)=O(B,\sigma)=\left\{z\in F(u)|
N_{F(u)/F}(z)=1\right\}$ if $B$ is not split.

\subsection{The Clifford Group}\label{ssec:Cliffor-Group}
In the twisted quadratic case, there is a Special Clifford group but
not the Clifford group and we recall this now.

Let $(A,\sigma)$ be a twisted quadratic space over a field $F$. We
have associated to $(A,\sigma)$ a subgroup
$\Gamma^+(A,\sigma)\subset C^+(A,\sigma)^*$ of the group of units of
the even Clifford algebra associated to the pair $(A,\sigma)$. If
$(A,\sigma)$ is split then $\Gamma^+(A,\sigma)$ can be identified
with the special clifford group of \cite{chevalley-spinors}.

We caution the reader that we use the notation $\Gamma^+(A,\sigma)$
instead of $\Gamma(A,\sigma)$ used in \cite{knus-book} because in
the split case we get the special clifford group using the above
construction (and not the Clifford group which we note, is denoted
by $\Gamma(V,q)$, in \cite{chevalley-spinors} while the special
Clifford group is denoted by $\Gamma^+(V,q)$). It seems to the
author that it is better to follow established conventions of the
theory in split case for psychological reasons.

We have an exact sequence of groups
\begin{equation}
1\to F^*\to \Gamma^+(A,\sigma) \to O^+(A,\sigma) \to 1
\end{equation}

\subsection{Recurring Example}\label{ssec:Recurri-Example4}
Let $(B,\sigma)$ be a twisted quadratic space, with $B/F$ a
quaternion division algebra
(see \ref{ssec:Recurri-Example}, \ref{ssec:Recurri-Example2}, \ref{ssec:Recurri-Example3}), and suppose
that $\sigma=\int(u)\circ\gamma$. Let $F(u)^{1}$ be the subgroup
\begin{equation}
F(u)^{1}=\left\{ z\in F(u): z\gamma(z)=1 \right\}
\end{equation}
Then we have and isomorphism of $F$-algebras
$$C^+(B,\sigma)=F(u),$$
and the group $O^+(B,\sigma)$ is given by
\begin{equation}
O^+(A,\sigma)=F(u)^1,
\end{equation}
while the special Clifford group $\Gamma^+(B,\sigma)$ can be identified with
$$\Gamma^+(B,\sigma)=C^+(B,\sigma)^*=F(u)^*.$$
We have an exact sequence of groups
\begin{equation}
1\to F^*\to \Gamma^+(B,\sigma)\to O^+(B,\sigma)\to 1.
\end{equation}

\section{The groups $\Gspin$ and $\Spin$}
\subsection{The group of proper
similitudes}\label{ssec:group-proper-similit}
From now on we will exclusively work with quaternion algebras over a
field $F$. In other words $A$ has $\deg(A)=2$.
%\subsection{}
Let $(A,\sigma)$ be a twisted quadratic space over
a field $F$ with $A$ a quaternion algebra. We will assume that $F$
is not of characteristic two for simplicity. Then we have a group
$\GO^+(A,\sigma)$. In the following subsections we will construct a
group, which we will denote by $\Gspin(A,\sigma)$ which bears a
 relation to $\GO^+(A,\sigma)$ similar to the relation the usual Spin group bears to
the special orthogonal group.

Let $\Gm/F$ be the multiplicative group over $F$. If $K/F$ is a
finite extension we will write $R_{K/F}\Gm$ for torus obtained by
the restriction of scalars. Let $N_{K/F}:K\to F$ be the norm map. We
have an induced map $N_{K/F}:R_{K/F}\Gm\to \Gm$ and let
\begin{equation} {\Gm^1}_{K/F}={\rm Ker}(N_{K/F}:R_{K/F}\Gm\to\Gm)
\end{equation}
 be its kernel. This is a group scheme whose group of $F$-rational points,
 $${\Gm^1}_{K/F}(F)=\left\{x\in K^* | N_{K/F}(x)=1 \right\},$$
is the subgroup of $K^*$ consisting of norm one elements in $K$.
Similarly define the group scheme
\begin{equation}
\mu^1_{K/F}={\rm Ker}(N_{K/F}:R_{K/F}\mu_2\to \mu_2).
\end{equation}

\begin{proposition}\label{pro:goplus}
Let $(A,\sigma)$ be a twisted quadratic space over a field $F$ and
assume that $A$ is a quaternion algebra. Let $K=C^+(A,\sigma)$. Then
the group scheme of proper similitudes and the group scheme of
proper isometries are given by
\begin{eqnarray*}
\GO^+(A,\sigma)&=&R_{K/F}\Gm.\\
\O^+(A,\sigma)&=&{\Gm^1}_{K/F}
\end{eqnarray*}
\end{proposition}
For a proof see \cite[Example 12.2.5, page 164]{knus-book}.

\subsection{Spin groups in the split
case}\label{ssec:Spin-groups-split-case} Let us recall a few
standard facts from the theory of spin groups. Let us assume that
$(V,q)$ be a quadratic space over $F$ with $q$ an isotropic form and
assume that $\dim(V)=2$.  In this case we have groups
$SO(q)=SO(2)=\Gm$. The Clifford algebra construction yields a spin
group as well. This situation is degenerate from the classical point
of view (because $SO(2)=\Gm$). We have ${\rm Spin}(2)=\Gm$ and we
also have an exact sequence of algebraic groups.
\begin{equation}
 1\to\mu_2\to {\rm Spin}(2)\to SO(2)\to 1
\end{equation}
This is none other than the Kummer sequence:
\begin{equation}
 1\to \mu_2\to\Gm\to\Gm\to1.
\end{equation}
We have in particular an exact sequence (a part of the Galois cohomology sequence for the Kummer sequence):
\begin{equation}
 1\to\mu(F)\to F^*\to F^*\to F^*/F^{*2}.
\end{equation}
The connecting homomorphism $F^*\to F^*/F^{*2}$ is the Spinor norm homomorphism $SO(2)(F)\to F^*/F^{*2}$.
\subsection{The group $\GSpin$}\label{ssec:GSpin}
We now want to describe a similar story for $\GO^+(B,\sigma)$ when $B$ is a quaternion algebra over $F$. For
notational simplicity, let $K=C^+(B,\sigma)$ be the Clifford algebra
associated to the twisted quadratic space $(B,\sigma)$. Then
$\dim_F(K)=2$. We have a canonical exact sequence of algebraic
groups (obtained by restriction of scalars):
\begin{equation}\label{res-scalars}
 1\to R_{K/F}\mu_2\to R_{K/F}\Gm\to R_{K/F}\Gm\to 1.
\end{equation}

Let $(B,\sigma)$ be a twisted quadratic space with $B$ a quaternion algebra. We define $\gspin(B,\sigma)$ as follows
\begin{equation}\label{eq:defngspin}
 \gspin(B,\sigma)=R_{K/F}\Gm.
\end{equation}
And we define the spin group $\spin(B,\sigma)$ by
\begin{equation}
\spin(B,\sigma)={\Gm^1}_{K/F}.
\end{equation}
 Thus we have the commutative diagram:
\begin{equation}
\xymatrix{
1 \ar[r] & R_{K/F}\mu_2 \ar@{=}[d] \ar[r] &R_{K/F}\Gm \ar@{=}[d]\ar[r] & R_{K/F}\Gm \ar@{=}[d]\ar[r] &1  \\
1 \ar[r] & R_{K/F}\mu_2 \ar[r] &\gspin(B,\sigma) \ar[r] &\go^+(B,\sigma) \ar[r] &1}
\end{equation}
And we have a commutative diagram of groups
\begin{equation}
 \xymatrix{
 1 \ar[r] & {\mu_2^1}_{K/F} \ar@{^{(}->}[d]  \ar[r] & \spin(B,\sigma) \ar@{^{(}->}[d]\ar[r] & \O^+(B,\sigma) \ar@{^{(}->}[d]\ar[r] &1  \\
1 \ar[r] & R_{K/F}\mu_2 \ar[r] &\gspin(B,\sigma) \ar[r]
&\go^+(B,\sigma) \ar[r] &1}
\end{equation}

Thus we have proved that
\begin{proposition}
Let $F$ be a field and $(B,\sigma)$ be a twisted quadratic space
over $F$ with $B$ a quaternion algebra over $F$. Let $K=C^+(B,\sigma)$ be the (even) Clifford algebra of $(B,\sigma)$.
Then there is a canonical isogeny $\gspin(B,\sigma)
\to\GO^+(B,\sigma)$ with kernel $R_{K/F}\mu_2$. We have a
commutative diagram of algebraic groups:
\begin{equation}
 \xymatrix{
 1 \ar[r] & {\mu_2^1}_{K/F} \ar@{^{(}->}[d]  \ar[r] & \spin(B,\sigma) \ar@{^{(}->}[d]\ar[r] & \O^+(B,\sigma) \ar@{^{(}->}[d]\ar[r] &1  \\
1 \ar[r] & R_{K/F}\mu_2 \ar[r] &\gspin(B,\sigma) \ar[r]
&\go^+(B,\sigma) \ar[r] &1}
\end{equation}
\end{proposition}

\section{Spin structures on an elliptic curve}\label{sec:Spin-structu-ellipti-curve}
\subsection{Elliptic curves of spinorial type}\label{ssec:Ellipti-curves-spinori-type}
Let $k$ be a field and let ${\bar k}$ be its algebraic closure. Let
$E/k$ be an elliptic curve. We will write $\End(E)=\Hom_k(E,E)$ for
the $\Q$-algebra endomorphisms of $E$ defined over $k$. We will say
that $E$ is of \textit{spinorial type} if the $\Q$-algebra
$\End_{k}(E)$ admits a non-trivial involution
of the first kind.

\begin{proposition}\label{pro:EhasSpin}
Let $E/k$ be an elliptic curve over a field $k$. Consider the
following assertions:
\begin{enumerate}
\item $E$ is of spinorial type.
\item $\End_{k}(E)=\End_{\bar k}(E)$ is a quaternion algebra.
\item $E$ is supersingular.
\item The characteristic of $k$ is $p>0$.
\end{enumerate}
Then we have $(1)\Leftrightarrow(2)\Rightarrow(3)\Rightarrow(4)$.
\end{proposition}

\begin{proof}
Clearly the assertions $(2)\Rightarrow(3)\Rightarrow(4)$ are
well-known. So we need to prove the equivalence
$(1)\Leftrightarrow(2)$.

Let $B=\End_{k}(E)$.  We prove $(1)\Rightarrow(2)$. By the
classification of the endomorphism algebras of $E$ there are three
possibilities: $B$ is either (a) the field of rational numbers (b)
an imaginary quadratic field or (c) a quaternion algebra. If $B=\Q$,
then $B$ does not admit any non-trivial involutions. Similarly if
$B$ is an imaginary quadratic field then $B$ does not admit any
non-trivial involutions of the first kind (\textit{i.e.},
involutions trivial on its center). Hence $B$ cannot be the field of
rational numbers or an imaginary quadratic field. So the hypothesis
$(1)$ implies that $B$ is a quaternion algebra.

Now to prove $(2)\Rightarrow(1)$ we use \cite[Corollary~2.8 (page 18) and Proposition~2.21 (page 26)]{knus-book}
which says that any quaternion algebra
$B/\Q$ admits non-trivial involutions of the first kind. So $E$ is
of spinorial type. This completes the proof.
\end{proof}

\subsection{Elliptic curves of spinorial type over finite fields}
%By Proposition~\ref{pro:EhasSpin} we see that any elliptic curve of
%spinorial type is definable over $\F_{p^2}$.
%So
From now on we will study elliptic curves of spinorial type over a
finite field $\F_q$ with $q=p^n$ elements. Before proceeding further
we recall the classification up to isogeny of elliptic curves over a
finite field $\F_q$ (see \cite{deuring41} or
\cite[Theorem~4.1]{waterhouse69}).
\begin{theorem}\label{waterhouse}
Let $\F_q$ be a field with $q=p^a$ elements. The set of isogeny classes of
elliptic curves over $\F_q$ are in bijection with a certain set of integers, denoted  $I_q$,
contained in the interval $[-2\sqrt{q},2\sqrt{q}]$. An integer
$\beta\in[-2\sqrt{q},2\sqrt{q}]$ is in $I_q$ if and only if:
\begin{enumerate}
 \item $(\beta,p)=1$; or
 \item $p|\beta$ and we are in  any of the following subcases:
 \begin{enumerate}
 \item $a$ is even and $\beta=\pm2\sqrt{q}$;
 \item $a$ is even and $p\not\equiv 1\mod{3}$ and $\beta=\pm\sqrt{q}$;
 \item $a$ is odd and $p=2,3$ and $\beta=\pm p^{\frac{a+1}{2}}$;
 \item $a$ is odd and $\beta=0$;
 \item $a$ is even and $p\not\congruent 1\mod{4}$ and $\beta=0$.
 \end{enumerate}
\end{enumerate}
In case $(1)$  the associated isogeny class consists of ordinary
elliptic curves. Otherwise the associated isogeny class consists of
supersingular elliptic curves. In all cases, except in the case
$2(a)$, the endomorphism algebra $\End_{\F_q}(E)$ of any elliptic
curve $E$ in the isogeny class associated to $\beta$ is an imaginary
quadratic field. If we are in the exceptional case $2(a)$ then the
endomorphism algebra is the unique quaternion algebra over $\Q$
which is ramified at $p$ and $\infty$.
\end{theorem}

\begin{proposition}\label{pro:spinorial}
Let $E/\F_q$ be a supersingular elliptic curve. Then $E\times_{\F_q}\F_{q^2}$ is of spinorial
type.
\end{proposition}
\begin{proof}
If $E$ is a supersingular elliptic curve defined over $\F_q$ then $E'=E\times_{\F_q}\F_{q^2}$ is
a supersingular elliptic curve over $\F_{q^2}$ with all its endomorphisms defined over
$\F_{q^2}$. Hence $E'$ is of spinorial type by the previous proposition.
\end{proof}

We end the subsection with the following consequence of
Proposition~\ref{waterhouse} and Proposition~\ref{pro:EhasSpin}.

\begin{proposition}\label{properties-of-spinorial}
Let $E/\F_q$ be an elliptic curve of spinorial type. Then
\begin{enumerate}
\item $E$ is supersingular,
\item The two eigenvalues of the Frobenius endomorphism of  $E$ are equal;
\item Frobenius endomorphism $\tau_E:E\to E$ is in the center of $B=\End_{\F_q}(E)$.
\end{enumerate}
\end{proposition}

\begin{proof}
This is immediate from Deuring's classification of supersingular
elliptic curves (see Theorem~\ref{waterhouse}).
\end{proof}

\subsection{Spin structures on an elliptic curve}
Let $E/\F_q$ be an elliptic curve over a finite field $\F_q$. Assume
that $E$ is an elliptic curve of spinorial type. Let $B=\End(E)$ be
the endomorphism algebra of $E$. A \textit{spin structure} on $E$ is
a choice of an involution $\sigma:B\to B$ of first kind and
orthogonal type. Equivalently, a spin structure on an elliptic curve
is a choice of twisted quadratic space structure $(B,\sigma)$  on
$B$ where $B=\End(E)$ and $\sigma:B\to B$ is an involution of the
first kind and orthogonal type.

\begin{proposition}\label{frobeniussimilitude}
Let $E/\F_q$ be an elliptic curve and suppose $E$ is an elliptic
curve with a spin structure $(B,\sigma)$. Then the Frobenius
endomorphism $\tau_E:E\to E$ induces a similitude of $(B,\sigma)$.
\end{proposition}
\begin{proof}
Let $E$ be an elliptic curve with a spin structure $(B,\sigma)$.
Then the Frobenius endomorphism $\tau_E:E\to E$ is in the center of
$B$. The center of $B$ is $\Q$. We have to prove that
$\sigma(\tau)\tau\in\Q$. But $\tau$ is in the center of $B$, so
$\sigma(\tau)=\tau$ and so $\sigma(\tau)\tau=\tau^2\in \Q$ as
$\tau\in\Q$. So $\tau$ is a similitude of $(B,\sigma)$. To prove
that $\tau$ is a proper similitude we have to prove that the
multiplier $\mu(\tau)=\sigma(\tau)\tau$ of the similitude $\tau$
satisfies,
$$\mu(\tau)=\sigma(\tau)\tau=\Nrd(\tau)=q.$$ The last equality
follows from the fact that $\tau$ operates by $\pm\sqrt{q}$ (or by
the following proposition, see \cite[Page 82]{mumford-abelian}). This
proves the assertion.
\end{proof}

\begin{proposition}
Let $E/\F_q$ be an elliptic curve of spinorial type. Let $\tau\in
\End(E)$ be the Frobenius endomorphism of $E$. Then we have
\begin{enumerate}
\item The reduced trace of $\tau$ is $\pm2\sqrt{q}$,
\item the reduced norm of $\tau$ is $q$.
\end{enumerate}
In particular, the reduced norm is a square.
\end{proposition}

\begin{proposition}\label{clifford-classes}
Let $(B,\sigma)/\Q$ be a twisted quadratic space with $B$ a quaternion algebra over $\Q$
which is ramified at $\infty$.  Let $K=C^+(B,\sigma)$ be the associated even Clifford algebra.
Then $K/\Q$ is an imaginary quadratic extension of $\Q$.
\end{proposition}
\begin{proof}
The algebra $B$ is  given by its symbol $B=\hasse{a}{b}{}$. Since $B$ is ramified at $\infty$,
we see that the Hilbert symbol $(a,b)_\infty=-1$ and this means that $a<0$ and $b<0$.
Now let $\sigma$ be an orthogonal involution on $B$ and let $\gamma$ be the canonical symplectic involution of $B$.
Then $\sigma$ is of the form $\sigma=\int(u)\circ\gamma$ for some $u\in B$ with $\gamma(u)=-u$.
Thus $u$ is a pure quaternion and $\disc(\sigma)=-\Nrd(u)$. But a simple calculation shows that
$$\Nrd(u)=-au_1^2-bu_2^2+abu_3^2,$$
where $u=iu_1+ju_2+ku_3$ and $i^2=a,j^2=b,ij=-ji=k$. Since $a<0$ and $b<0$,
we see that the quadratic form $(-a,-b,ab)$ is positive definite. Hence $\Nrd(u)>0$ and
in particular $\disc(\sigma)=-\Nrd(u)<0$. By Theorem~\ref{center-and-disc2}
we know that in this case $K=C^+(B,\sigma)=\Q[x]/(x^2-\disc(\sigma))$ and as $\disc(\sigma)<0$, $K$
is an imaginary quadratic extension of $\Q$. This completes the proof.
\end{proof}

\section{Spinorial representation of the Weil group}
\subsection{Weil groups}
Let $E/\F_q$ be an elliptic curve. Let $W(\bF_q/\F_q)\subset
\gal(\bF_q/\F_q)$ be the Weil group of $\F_q$. It is standard that
$\Z\isom W(\bF_q/\F_q)$ and we  choose the isomorphism given by $1\mapsto \frob_{geom}$,
where $\frob_{geom}\in\gal(\bF_q/\F_q)$ is the geometric Frobenius of $\F_q$.
Proposition~\ref{frobeniussimilitude} provides a canonical
representation of the Weil group arising from $(E,(B,\sigma))$:
\begin{proposition}\label{pro:similitude-rep} Let $E$ be an elliptic curve with a spin structure
$(B,\sigma)$. Then there is a canonical similitude representation
\begin{equation}
  \rho_{E,\sigma}: W(\bF_{q}/\F_q) \to \go^+(B,\sigma),
\end{equation}
which is given by
$$\rho_{E,\sigma}(\frob_{geom})=(\tau,\tau)\in\go^+(B,\sigma)(K)=R_{K/\Q}\Gm(K)=K^*\times K^*.$$
\end{proposition}
\subsection{Spinorial liftings}
 Let $\rho:W(\bF_q/\F_q)\to\go^+(B,\sigma)$ be a homomorphism of groups.
 We say that $\rho$ admits a spinorial lifting if there exists a
 representation $\rho':W(\bF_q/\F_q)\to \gspin(B,\sigma)$ which makes the following diagram commutative:
\begin{equation}
 \xymatrix{
& \gspin(B,\sigma)\ar[d]\\
W(\bF_q/\F_q)\ar[ur]^{\rho'}\ar[r]^\rho&\go^+(B,\sigma)
 }
\end{equation}

\subsection{A criterion for existence of Spinorial liftings}
\begin{theorem}\label{existence-criterion}
 Let $E/\F_q$ be an elliptic curve with a spin structure $(B,\sigma)$.
 Let $K=C^+(B,\sigma)$ be the associated even Clifford algebra.
 Let $\tau_E=\tau\in \End(E)$ be the Frobenius endomorphism of $E$. Then a lift $\rho_E^{spin}:W(\bF_q/\F_q)\to
 \gspin(B,\sigma)$ exists if and only if the Clifford algebra $K=K(B,\sigma)$ satisfies $K=F(\sqrt{\tau})$.
\end{theorem}
\begin{proof}
 The endomorphism $\tau\in \End(E)=B$ is a central element in $B$ with reduced norm $\Nrd(\tau)=q$.
 So by Theorem~\ref{waterhouse} $\tau=\pm \sqrt{q}$. We have an exact sequence of groups obtained
 by taking Galois cohomology of \eqref{res-scalars} with $G_\Q=\gal({\bar \Q}/\Q)$, we get (for the Galois cohomology computations, which are easy, the unfamiliar reader may use \cite[Lemma 29.6, page 394]{knus-book}):
\begin{equation}
 1\to \mu_2(K)\to K^*\to K^*\to K^*/K^{*2}=H^1(G_\Q,R_{K/\Q}(\mu_2)).
\end{equation}
Then $\tau\in
\go^+(B,\sigma)$ lifts to
$\gspin(B,\sigma)$ if and
only if its image in
$K^*/K^{*2}$ is $1$.
Equivalently a lifting
exists if and only if
$\sqrt{\tau}\in K$. When
this happens, we have a
lift
$$\rho^{spin}:W(\bF_q/\bF_q)\to
\gspin(B,\sigma),$$ given by
$$\rho^{spin}(\frob_{geom})=(\sqrt{\tau},-\sqrt{\tau})\in\gspin(B,\sigma)(K)=K^*\times
K^*.$$  This prove the proposition.
\end{proof}

\section{Arithmetic spin structures}
\subsection{Definition of Arithmetic spin structures}
Let $E/\F_q$ be an elliptic curve with a spin structure
$(B,\sigma)$.  We say that the spin structure $(B,\sigma)$ is an
\textit{arithmetic spin structure} if the canonical similitude
representation $\rho_{E,\sigma}:W(\bF_q/\F_q)\to \go^+(B,\sigma)$
admits a spinorial lifting.
\subsection{Existence of an arithmetic spin structure} We now show that
arithmetic spin structures exists on an elliptic curve under suitable circumstances.
\begin{theorem}\label{the:existenceofarithmetic}
 Let $E/\F_q$ be an elliptic curve of spinorial type with $q=p^{2n}$. Let $B=\End(E)$ and
 let $\tau\in B$ be the Frobenius endomorphism of $E$.
 Then we have the following:
\begin{enumerate}
 \item if $n$ is odd then there exists a unique arithmetic spin structure of
 discriminant $-p^n$ on $E/\F_q$ if and only if $\tau\in \End(E)$ is given by multiplication by $-p^n$.
\item if $n$ is even, then there exists an arithmetic spin structure of $E$ if and
only if $B$ contains a pure quaternion with reduced norm $1$.
\end{enumerate}
\end{theorem}
\begin{proof}
We prove (1). If $E$ carries an arithmetic spin structure $\sigma$
of discriminant $-p^n$, then by Theorem~\ref{existence-criterion} we
see that $\tau^2\in C^+(B,\sigma)$ and
$\Q(\sqrt{\tau})=C^+(B,\sigma)=\Q[x]/(x^2+p^n)$ thus $\tau=-p^n$. So
we have to prove the converse.

To produce an arithmetic spin structure we have to prove the
existence of a spin structure $(B,\sigma)$ on $E/\F_q$ such that the
canonical similitude representation
$$\rho_{E,\sigma}:W(\bF_q/\F_q)\to \go^+(B,\sigma)$$ lifts to a spin
representation. By Theorem~\ref{existence-criterion}, this happens
if and only if $\tau=\pm p^n\in K^{*2}$. By
Proposition~\ref{clifford-classes} we know that the  discriminants
of involutions which can occur as spin structures are all negative.
So we see that there is no orthogonal involution on $B$ whose
discriminant can be $+p^n$. So it remains to show that there is an
involution $\sigma$ whose discriminant is $-p^n$.
 To construct an involution of the first kind and of orthogonal
type on $B$ with discriminant $-p^n$ it suffices to construct a pure quaternion $u\in B$
(i.e. a quaternion with reduced trace zero) with reduced norm $\Nrd(u)=p^n$.
Indeed given such a $u$, the discriminant of
the involution $\sigma=\int(u)\circ\gamma$ is given by (see
\ref{ssec:Discrim})
$$\disc(\sigma)=-\Nrd(u)=-p^n.$$ Thus we see that the even Clifford algebra
$C^+(B,\sigma)$ is $C^+(B,\sigma)=K=\Q[x]/(x^2+p^n)$. Further
orthogonal involutions on $B$ are classified, up to isomorphism, by
their discriminants (by \cite[7.4, Page 82]{knus-book} any two orthogonal involutions differ by an inner conjugation by a non-zero quaternion). So the spin structure $(B,\sigma)$ is unique up to isomorphism and is arithmetic.

So let us construct the required $u$. We claim now that there exists
a quaternion $u\in B$ of trace zero, with reduced norm $p^n$, for
any odd value of $n$. Indeed it is easy to see that $B$ can be given
by symbols $\hasse{-a}{-p}{}$ for a suitable choice of $a$ and so
$B$ contains a trace zero quaternion $v$ such that $v^2=-p$. The
reduced norm of $v$ is $\Nrd(v)=p$ and hence the quaternion
$u=p^{(n-1)/2}v$ has reduced norm $p^n$. Thus the involution
$\sigma=\int(u)\circ\gamma$ has discriminant
$$\disc(\sigma)=-\Nrd(u)=-p^n$$
as claimed.  Thus the Clifford algebra $K=\Q[x]/(x^2+p^n)$, and so
$\sqrt{-p^n}\in K$ and so we see that if $E$ is such that
$\tau=-p^n$, then $\tau$ has a square root in $K$ and so we have a
lifting
\begin{equation}
\rho^{spin}_{E,\sigma}:W(\bF_q/\F_q)\to \gspin(B,\sigma)
\end{equation} given by $$\rho^{spin}_{E,\sigma}(\frob)=(\sqrt{\tau},-\sqrt{\tau})\in\go^+(B,\sigma)(K)=K^*\times K^*.$$

Now we prove (2). If $n$ is even, then $\tau=-p^n$ has a square root
in $K$ if and only if $\sqrt{-1}\in K$. So  the existence
of an involution with this property is tantamount to finding a pure
quaternion $v$ whose reduced norm is $1$. If we can find such a
$v$, then we can take $u=p^{n/2}v$ and observe that $u$ has reduced norm
$\Nrd(u)=p^n$. Hence $\sigma=\int(u)\circ\gamma$ is an orthogonal
involution of the first kind with discriminant $\disc(\sigma)=-p^n$ and
$\tau=-p^n$ has a square root in $K=\Q[x]/(x^2+p^n)=\Q(i)$.
\end{proof}

\begin{corollary}\label{cor:existenceoverfp2}
Let $E/\F_{p^2}$ of spinorial type. Then $E$ carries an arithmetic
spin structure if and only if the Frobenius endomorphism of $E$
operates by multiplication by $-p$. If $E$ has an arithmetic spin
structure, then it is unique up to isomorphism.
\end{corollary}

\section{A candidate for $\Q(1/4)$}
\subsection{Definition of $\Q(1/4)$} Let $E/\F_{p^2}$ be an elliptic curve with an
arithmetic spin structure. We will define $\Q(1/4)$ to be the triple
$(E,(B,\sigma))$ where $E$ is our elliptic curve and $B=\End(E)$ and
$(B,\sigma)$ is the arithmetic spin structure on $E$. We show now
that $\Q(1/4)$ is equipped with an $\ell$-adic (for $\ell\neq p$)
and a crystalline realization at $p$.
\subsection{The $\ell$-adic realization of $\Q(1/4)$}\label{ell-realization}
The object $\Q(1/4)$ comes equipped with an $\ell$-adic realization,
denoted $\Q(1/4)_\ell$. We define the $\ell$-adic realization
$\Q(1/4)_\ell$ of $\Q(1/4)$ to be the representation obtained by
extension of scalars of the canonical spin representation
$\rho^{spin}_{E,\sigma}\to\gspin(B,\sigma)$. This gives us a Weil
sheaf (by \cite{deligne80} such a sheaf is specified by specifying a
representation of the Weil group) on the point ${\rm
Spec}(\F_{p^2})$:
\begin{equation}
\rho^{spin}_{E,\sigma,\ell}:W(\bF_{p^2}/\F_{p^2})\to
R_{K/\Q}\Gm(\Q_\ell)=\gspin(B,\sigma)(\Q_\ell).
\end{equation}
\subsection{The crystalline realization of $\Q(1/4)$}\label{p-realization}
Let $(E,(B,\sigma))$ be an elliptic curve over $\F_{p^{2n}}$ with an
arithmetic spin structure $(B,\sigma)$. Let $W=W(\F_{p^{2n}})$ be
the ring of Witt vectors of $\F_{p^{2n}}$. Let $F:W\to W$ be the
Frobenius morphism of $W$. Let  $K_0=W\tensor\Q_p$ be the quotient field of
$W$ and let $F:K_0\to K_0$ be the extension of $F$ to $K_0$. Let
$H^1_{cris}(E/W)$ be the crystalline cohomology of $E$; we have an
$F$-linear map $\phi:H^1_{cris}(E/W)\to H^1_{cris}(E/W)$. The data $(H^1_{cris}(E/W),\phi)$ is the data of an $F$-crystal associated to $E/\F_{p^{2n}}$. Let
$(M,\phi)$ denote the extension of the $F$-crystal
$(H^1_{cris}(E/W),\phi)$ to an $F$-isocrystal over $K_0$; since $E$ has an arithmetic spin structure, the $F$-linear
map $\phi:M\to M$ is simply the map $\phi=-p^nF$. We want to construct a crystal
$(M^{spin},\phi^{spin})$ which we would like to call the crystalline
realization of $\Q(1/4)$. The endomorphism algebra of $(M,\phi)$ is
the unique quaternion algebra $B\tensor \Q_p$ over $\Q_p$. The
orthogonal involution  $\sigma:B\to B$ extends to an orthogonal
involution $\sigma_p:B\tensor\Q_p\to B\tensor\Q_p$. The data
$(M,\phi,\sigma_p)$ is an $F$-isocrystal equipped with a spin
structure (in the obvious sense). The Clifford algebra
$C^+(M,\phi,\sigma_p)$ is a $\Q_p$ vector space isomorphic to
$\Q_p[x]/(x^2+p^n)$. Define $M^{spin}=C^+(M,\phi,\sigma_p)$, and
define $\phi^{spin}$ to be multiplication by $x\in \Q_p[x]/(x^2+p^n)$
thought of as an $F$-linear endomorphism of $M^{spin}$. This gives
us a Dieudonne isocrystal $(M^{spin},\phi^{spin})$ over $\Q_p$ which
we call the crystalline realization of $\Q(1/4)$ and denote it by
$\Q(1/4)_p$.
\subsection{The $L$-function of $\rho^{spin}_{E,\sigma}$}
Let $E/\F_q$ be an elliptic curve with an arithmetic spin structure
$(B,\sigma)$. Then $q=p^{2n}$ and the Frobenius endomorphism of $E$
operates by $-p^n=-\sqrt{q}$; the eigenvalues of Frobenius under the
spin representation $\rho^{spin}_{E,\sigma}$ are $\pm
\sqrt{-p^n}=\pm\sqrt{-\sqrt{q}}$. Thus the Hasse-Weil zeta function of
$\rho^{spin}_{E,\sigma}$ is given by
\begin{equation}
Z(\rho^{spin}_{E,\sigma},T)=\frac{1}{\left(1-\sqrt{-p^n}T\right)\left(1+\sqrt{-p^n}T\right)}.
\end{equation}
We will write
\begin{equation}
L(\rho^{spin}_{E,\sigma},s)=Z(\rho^{spin}_{E,\sigma},q^{-s}).
\end{equation}
And this is
\begin{equation}\label{lrho}
L(\rho^{spin}_{E,\sigma},s)=\frac{1}{\left(1+q^{\frac{1}{2}-2s}\right)}.
\end{equation}
We note that \eqref{lrho} can also be written more suggestively as
\begin{eqnarray}\label{lrho2}
L(\rho^{spin}_{E,\sigma},s)&=&\frac{1}{\left(1+q^{\frac{1}{2}-2s}\right)}\\
&=&\frac{1}{\left(1+{q^2}^{(\frac{1}{4}-s)}\right)}\\
&=&\frac{1}{\left(1+\sqrt{-1}q^{\frac{1}{4}-s}\right)\left(1-\sqrt{-1}q^{\frac{1}{4}-s}\right)},
\end{eqnarray}
especially the last equality which makes the $1/4$-twisting more transparent.
Note that the presence of the two factors in the last equality should be seen as a manifestation
of the fact that the Clifford algebra $C^+(B,\sigma)$ is a quadratic extension of $\Q$.
\subsection{The $L$-function of $H^1(E)$}
Since the Frobenius endomorphism of $E$ operates by $-\sqrt{q}$, we
can also calculate the $L$-function  $E$, that is the reciprocal of
the characteristic polynomial of Frobenius on $H^1(E,\Q_\ell)$. The
characteristic polynomial  is given by
\begin{equation}\label{zetaE}
Z(H^1(E),T)=\left(1+\sqrt{q}T\right)^2.
\end{equation}
The $L$-function of $E$ is defined as
\begin{equation}\label{lfuncte}
    L(E,s)=Z(H^1(E),q^{-s})^{-1}=\frac{1}{\left(1+q^{\frac{1}{2}-s}\right)^2}
\end{equation}
\subsection{A relation between $L(\rho^{spin}_{E,\sigma},s)$ and $L(H^1(E),s)$}
From \eqref{lrho} and \eqref{zetaE} we  have the relation:
\begin{theorem}\label{l-identities}
For an elliptic curve $E/\F_q$ with an arithmetic spin structure
$(B,\sigma)$ and spinorial representation
$\rho^{spin}_{E,\sigma}:W(\bF_q/\F_q)\to \gspin(B,\sigma)$, we have
$$L(E,s)=L\left(\rho^{spin}_{E,\sigma},\frac{s}{2}\right)^2.$$
In particular we have
$$L(E,1)=L\left(\rho^{spin}_{E,\sigma},\frac{1}{2}\right)^2.$$
\end{theorem}

\bibliographystyle{mrl}
%\bibliography{spin}

\begin{thebibliography}{10}

\bibitem{anderson86}
G.~W. Anderson, \emph{Cyclotomy and an extension of the {T}aniyama group},
  Compositio Math. \textbf{57} (1986), no.~2,  153--217.

\bibitem{chevalley-spinors}
C.~C. Chevalley, The algebraic theory of spinors, Columbia University Press,
  New York (1954).

\bibitem{deligne72}
P.~Deligne, \emph{La conjecture de {W}eil pour les surfaces {$K3$}}, Invent.
  Math. \textbf{15} (1972) 206--226.

\bibitem{deligne80}
---{}---{}---, \emph{La conjecture de {W}eil. {II}}, Inst. Hautes \'Etudes Sci.
  Publ. Math.  (1980), no.~52,  137--252.

\bibitem{deninger94}
C.~Deninger, \emph{Motivic {$L$}-functions and regularized determinants. {II}},
  in Arithmetic geometry ({C}ortona, 1994), Sympos. Math., XXXVII, 138--156,
  Cambridge Univ. Press, Cambridge (1997).

\bibitem{deuring41}
M.~Deuring, \emph{Die {T}ypen der {M}ultiplikatorenringe elliptischer
  {F}unktionenk\"orper}, Abh. Math. Sem. Hansischen Univ. \textbf{14} (1941)
  197--272.

\bibitem{knus-book}
M.-A. Knus, A.~Merkurjev, M.~Rost, and J.-P. Tignol, The book of involutions,
  Vol.~44 of \emph{American Mathematical Society Colloquium Publications},
  American Mathematical Society, Providence, RI (1998). With a preface in
  French by J.\ Tits.

\bibitem{kuga67}
M.~Kuga and I.~Satake, \emph{Abelian varieties attached to polarized
  {$K\sb{3}$}-surfaces}, Math. Ann. \textbf{169} (1967) 239--242.

\bibitem{manin92}
Y.~Manin, \emph{Lectures on zeta functions and motives (according to {D}eninger
  and {K}urokawa)}, Ast\'erisque  (1995), no. 228,  4, 121--163. Columbia
  University Number Theory Seminar (New York, 1992).

\bibitem{mumford-abelian}
D.~Mumford, Abelian varieties, Tata Institute of Fundamental Research Studies
  in Mathematics, No. 5, Published for the Tata Institute of Fundamental
  Research, Bombay (1970).

\bibitem{ramachandran05}
N.~Ramachandran, \emph{Values of zeta functions at {$s=1/2$}}, Int. Math. Res.
  Not.  (2005), no.~25,  1519--1541.

\bibitem{waterhouse69}
W.~C. Waterhouse, \emph{Abelian varieties over finite fields}, Ann. Sci.
  \'Ecole Norm. Sup. (4) \textbf{2} (1969) 521--560.

\end{thebibliography}

\end{document}